\documentclass[12pt]{amsart}

\def \R{\mathbb{R}}

\newtheorem{theorem}{Theorem}[section]
\newtheorem{coro}[theorem]{Corollary}

\newtheorem{lema}[theorem]{Lemma}
\newtheorem{rema}[theorem]{Remark}
\newtheorem{rem}[theorem]{Remark}

\newtheorem{ejem}[theorem]{Example}

\newtheorem{prop}[theorem]{Proposition}
\newtheorem{teor}[theorem]{Theorem}

\usepackage{url,hyperref}
\hypersetup{}

\usepackage{euscript,color}

\definecolor{red}{rgb}{1,0,0}

\hypersetup{
    bookmarks=true,         
    unicode=false,          
    pdftoolbar=true,        
    pdfmenubar=true,        
    pdffitwindow=false,     
    pdfstartview={FitH},    
    pdftitle={My title},    
    pdfauthor={Author},     
    pdfsubject={Subject},   
    pdfcreator={Creator},   
    pdfproducer={Producer}, 
    pdfkeywords={keyword1} {key2} {key3}, 
    pdfnewwindow=true,      
    colorlinks=true,       
    linkcolor=black,          
    citecolor=green,        
    filecolor=magenta,      
    urlcolor=blue         
}

\def\d{\overrightarrow{\mathrm{d}}}

\def\R{\mathbb{R}}
\def\dif{\mathrm{d}}
\def\H{{\mathrm{H}}}
\def\Di{{\mathrm{D}}}
\def\No{{\mathrm{N}}}
\def\Tr{{\mathrm{Tr}}}
\def\adj{{\mathrm{adj}}}
\def\det{{\mathrm{det}}}

\begin{document}

\title[Minimal helix submanifolds and Minimal Riemannian foliations]
{Minimal helix submanifolds and Minimal Riemannian foliations}

\author[A. J. Di Scala]{Antonio J. Di Scala}
\author[G. Ruiz-Hern\'andez]{Gabriel Ruiz-Hern\'andez}

\thanks{This research was partially supported by Ministero degli Affari Esteri from Italy and CONACYT from Mexico. }

\date{\today}

\maketitle

\begin{abstract}
We investigate minimal helix submanifolds of any dimension and codimension immersed in Euclidean space.
Our main result proves that a ruled minimal helix submanifold is a cylinder. As an application we classify
complex helix submanifolds of $\mathbb{C}^n$:  They are extrinsic products with a complex line as a factor.
The key tool is Corollary \ref{offsets} which allows us to classify Riemannian foliations of
open subsets of the Euclidean space with minimal leaves.
Finally, we consider the case of a helix hypersurface with constant mean curvature and prove that it is
either a cylinder or an open part of a hyperplane.
\end{abstract}

\section{Introduction}

A submanifold $M \subset \mathbb{R}^{n}$ is called a \emph{helix} with respect to $\d \in \R^{n}$ if the angle \[ \theta(p) := \angle(T_pM, \d)\]  between the tangent space $T_pM$ and a fixed direction $\d \in \R^n$ is constant, i.e. $\theta(p)$ does not depend upon $p \in M$. Observe that the angle $\theta(p)$ is related to the splitting $\d = \d^{\top} + \d^{\perp}$ according to the tangent and normal components of $\d$ at $p \in M$.
Indeed,  the norm $\|\d^{\top}\|$ at $p \in M$ is given by $\| \d\|  \mathrm{cos}(\theta(p))$. Then $M \subset \mathbb{R}^{n}$ is a helix with respect to $\d$ if and only if the norm $\|\d^{\top}\|$ is constant along $M$.
Observe that $\d^{\top}$ is the gradient of the height function $h_{\d}(x) := \langle x, \d \rangle$ by \cite[Proposition 4.1.1, page 65]{PT88}. So $M$ is a helix with respect to $\d$ if and only if the height function $h_{\d}$ is a so called eikonal function i.e. the norm of its gradient $\nabla_M h_{\d}$ is constant on $M$.\\

In this paper we are interested in the local geometry of the helix $M$ i.e. all the claims are of local nature unless otherwise specified.
Important examples of helix submanifolds are totally geodesic submanifolds of shadow boundaries.
We refer to \cite{Ghomi} and \cite{RH} for details. Helix submanifolds are also called constant angle submanifolds and had been studied in other ambient spaces, see for example
\cite{Di} and \cite{Mu}.

Let us briefly recall the two methods to study helix submanifolds that were developed in \cite{mona}, \cite{kodai}.
Namely, the {\bf projection method} and the {\bf slice method}.

The {\bf projection method} considers the helix $M$  as the graph of a function $f$ defined on the projection $B$ of $M$ to an hyperplane $H$ orthogonal to $\d$. More precisely, let $M \subset \mathbb{R}^{n}$ be a helix submanifold of angle $\theta \notin \{ 0,\frac{\pi}{2}\}$ with respect to the unit vector $\d$.
Let $\pi : \R^n \rightarrow H$ be the orthogonal projection to an hyperplane $H$ orthogonal to $\d$.
The restriction of $\pi$ to $M$ is an immersion and $B = \pi(M)$ is called the base of the helix $M$.
Then $M$ looks locally as the graph of a function $f : U \subset B \rightarrow \R$. That is to say,
$M$ is locally the image of the map $ \phi: B \rightarrow \R^n = H \times \R $ defined as
\begin{equation}
\label{eqn:projection-method}
 \phi(p) := (i(p) , f(p))
\end{equation}
where $i$ is the canonical inclusion of $\pi(M)$.

Conversely we can start from a submanifold $B \subset H$ and a function $f \in C^{\infty}(B)$ and construct $M \subset \R^n$ as the graph of $f$ (see Theorem \ref{teor:eikonalvshelix}).\\

The {\bf slice method} can be used when the helix is ruled, i.e. the integral curves of $T:= \frac{\d^{\top}}{\|\d^{\top}\|}$ are geodesics in the Euclidean space
\cite[page 194, Definition 2.3]{kodai}. Let us briefly recall the local structure of a ruled helix, for more details see \cite[Theorem 4.6, page 202]{kodai}. Let $L = M \bigcap H$ be a slice of $M$ where $H$ is a hyperplane perpendicular to $\d$. Observe that $T$ is a normal vector field of $L = M \bigcap H$.

If the helix is ruled then $M$ is the union of the parallel manifolds $L_{s T}$ to $L$ in the $T$-direction.
Namely, $M$ is the image of the map \[ e: L \times (-\epsilon,\epsilon) \rightarrow \R^n \] defined as
\[ e(p,s) := p + s T(p) \, .\]

The projection and slice methods are related via the height function $h_{\d}$ in the following way.
Let $M$ be the helix submanifold, $B \subset H$ be its base and $f: B \rightarrow \R$ be the eikonal function as explained above. Then $h_{\d} = f \circ \pi$ hence  $\nabla_M h_{\d}$ is parallel to the vector field $T$.
Therefore the slices of $M$ with hyperplanes orthogonal to $\d$ are the parallel submanifolds $L_{s T}$.

\vspace{0.5cm}

Here is the main result of this paper.

\begin{teor}
\label{maintheorem}
If $M \subset \mathbb{R}^n$ is a full minimal ruled helix with respect to $\d \in \R^{n}$ then $\d$ is tangent to $M$.
That is to say, the helix angle $\theta$ is zero and $M$ is a cylinder over a minimal submanifold contained in a hyperplane $H$ orthogonal to $\d$.
\end{teor}

We do not know if the hypothesis of being {\bf ruled} can be omitted in the above statement.\\

Then we obtain the classification of complex helix submanifolds of $\mathbb{C}^n$.

\begin{teor}\label{complex} Let $M^{m} \subset \mathbb{C}^n$ be a full
complex submanifold of complex dimension $m$.
Assume that $M$ is a helix of angle $\theta$ with respect to a direction
$\d \in \mathbb{C}^n$. Then $\theta = 0$ and so
$M$ is locally an extrinsic product \[ M = \mathbb{C} \times N \subset
\mathbb{C} \times \mathbb{C}^{n-1} \, ,\]
where $N \subset \mathbb{C}^{n-1}$ is a complex submanifold.
\end{teor}

It is important to notice that the above theorem is not a direct consequence of Theorem \ref{maintheorem}
since we do not assume the complex helix submanifold to be ruled.\\

The main tool to prove the above theorems is Lemma \ref{LA} which we think is interesting in itself.
Indeed, in Submanifold Geometry \cite{Olmos} it is well-known that if the parallel manifolds $M_{t \xi}:= M + t\xi \subset \R^n$ in the direction of a normal parallel vector field $\xi$
are minimal submanifolds for small values of $t$ then $\xi$ is constant in $\R^n$.
We show that this is still true just assuming that $\xi$ has constant length (i.e. the hypothesis on $\xi$ of being normal parallel is not necessary).
Namely, we have the following corollary of Lemma \ref{LA}.

\begin{coro}
\label{offsets}
Let $M \subset \mathbb{R}^n$ be a submanifold and let $\xi \in \Gamma(\nu(M))$ be a normal vector field of constant length i.e. $\| \xi \| = constant $.
If the submanifolds $M_{t \xi} := M + t \xi \subset \R^n$
are minimal submanifolds for small values of $t$ then $\xi$ is constant in $\R^n$, i.e. $\xi$ is parallel with respect the normal connection and $A_{\xi} \equiv 0$, where $A_{\xi}$ is the shape operator of $M$ in direction $\xi$.
\end{coro}

The above corollary have the following interesting application to Riemannian foliations of the Euclidean space.
In \cite[page 450]{Murphy} the author wrote

{\it... it is easy to construct non-trivial examples of regular complex Riemannian foliations in $\mathbb{C}^n$ of all codimensions. (sic) \rm}

Indeed, the totally geodesic foliation given by the family of parallel affine subspaces $\{ \mathbf{V} + p \}$, $p \in \mathbf{V}^{\perp}$
to a fixed vector subspace $\mathbf{V} \subset \mathbb{C}^n$ give such examples.
The following theorem shows that they are (even locally) the unique examples.

\begin{teor}\label{MinimalRiemannianFoliation} Let $\mathcal{F}$ be a Riemannian foliation of an open subset $U$
of $\mathbb{R}^n$ with minimal leaves i.e. any leave of $\mathcal{F}$ is a minimal submanifold of $\mathbb{R}^n$.
Then $\mathcal{F}$  is totally geodesic. More precisely, for each $p \in U$ there is a neighborhood $G$ of $p$
such that the leaves of the restriction $\mathcal{F}|_{G}$ are open subsets of a foliation of $\mathbb{R}^n$ by parallel affine subspaces.
In particular, any complex Riemannian foliation of an open subset of $\mathbb{C}^n$ is totally geodesic.
\end{teor}

In section \ref{general} we give general results and discuss some interesting examples about (non necessarily ruled) minimal helices and its intrinsic geometry.\\

Finally we give the following generalization of a result in \cite{GPR}.

\begin{teor}\label{meancurvature} A helix hypersurface  $M^{n} \subset \R^{n+1}$ with constant mean curvature is either a cylinder $M = \R \times N \subset \R \times \R^{n} = \R^{n+1}$ over a hypersurface $N \subset \R^{n}$  with constant mean curvature or an open subset of a hyperplane i.e. $M$ is a totally geodesic hypersurface of $\R^{n+1}$.
\end{teor}

The above theorem is a special case of \cite[Theorem 15]{GPR2} where a similar result valid for hypersurfaces of products $\mathbb{R} \times N$ is obtained by using Bochner's formula.
Instead our proof is based in Ruh-Vilms's theorem \cite{RuVi} and a maximum principle for harmonic maps due to Sampson \cite[Theorem 2]{Sam}.
We also explain why our proof can not be extended to the case of higher codimensional minimal helix submanifolds.

\section{Minimal ruled helices}

The following result proved in \cite[Theorem 3.4, page 211]{mona} is going to play a key role along this paper.
\begin{teor}\cite[Theorem 3.4, page 211]{mona}
\label{teor:eikonalvshelix}
In the above notation, the submanifold $M$ is a helix if and only if $f$ is an eikonal function of $B$, i.e.
$\| \nabla_B f \|$ is constant on $B$. Here $\nabla_B f$
is the gradient of $f$ with respect to the induced metric on $B$ from $H \subset \R^n$.
\end{teor}

Let $L$ be an immersed $l-$dimensional submanifold in $\mathbb{R}^ n$, let $\eta$ be a normal vector field to $L$ of
constant length. The shape operator $A_{\eta}$ of $L$ in direction $\eta$ is given by \[ A_{\eta}(X) = - (D_{X} \eta)^{\top}\, \, ,\] where $D$ is the directional derivative of $\R^n$.
Let $L_\eta$ be the parallel submanifold given by the immersion $t_\eta(p) = p+\eta(p)$ where $p \in L$ (cf. \cite{Olmos} page 117). We also assume that $1$ is not an eigenvalue of $A_{\eta}$.

\begin{lema}
\label{lema:tangent-and-normalframes}
Let $E_i$, $1 \leq i \leq l$ be an orthonormal local frame of $TL$ such that $A_\eta(E_i)= \lambda_i E_i$, i.e.
this frame diagonalize the shape operator $A_\eta$ of $L$ in direction $\eta$.
Let $\xi_j$, $1 \leq j \leq n-l$ be a local orthonormal frame of $\nu L$ the normal bundle of $L$.
Then the corresponding tangent $X_i$ and normal $\tilde{\xi_j}$ frames of $L_\eta$ are given by

\[ \begin{cases} X_i= (1-\lambda_i) E_i + \nabla^\perp_{E_i} \eta \, ,\\  \tilde{\xi_j} = \xi_j
- \sum_{k=1}^{l} \frac{ \langle \nabla^\perp_{E_k} \eta, \xi_j \rangle}{1-\lambda_k} E_k \, . \end{cases} \]

In particular, the metric $G=(G_{rs})$ of $L_\eta$ with respect to the frame $X_i$'s is given by
\begin{equation}\label{metric} G_{rs} = (1-\lambda_r)(1-\lambda_s)\delta_{rs} + \langle \nabla^\perp_{E_r} \eta, \nabla^\perp_{E_s} \eta   \rangle .
\end{equation}
\end{lema}
\begin{proof}
The vectors fields $X_i$'s are tangent to $L_\eta$ because
$$X_i =(t_\eta)_*(E_i) = D_{E_i}(p + \eta(p)) = E_i + D_{E_i} \eta = E_i + \nabla^\perp_{E_i} \eta - A_\eta(E_i) .$$
Let us see that the vectors fields $\tilde{\xi_j}$ are orthogonal to the $X_j$'s:
$$\langle \tilde{\xi_j}, X_i \rangle = \langle \nabla^\perp_{E_i} \eta, \xi_j \rangle -
\sum_{k=1}^{l} \frac{1- \lambda_i}{1-\lambda_k}   \delta_{ik} \langle \nabla^\perp_{E_k} \eta, \xi_j \rangle
=0.$$
\end{proof}

Let $M \subset \mathbb{R}^n$ be helix with respect to the direction $\d \in \mathbb{R}^n$.
Let $\pi: \mathbb{R}^n \rightarrow H$ be the projection to a normal hyperplane $H$ to $\d$.

\begin{prop}
\label{decomposition}
Let $M \subset \R^n $ be a full minimal ruled helix and let $B=\pi(M) \subset H$ be its base.
Let $L = M \cap H \subset \pi(M)=B$ be a slice. Let $\eta := T$ be the restriction of $T$ to the slice $L$.
Then either $M$ is a cylinder over a submanifold of $H$ or it is the union of the $\eta$-parallel
manifolds $L_{s\eta}$ to $L$ which are minimal submanifolds of hyperplanes parallel to $H$.
\end{prop}
\begin{proof}
Let us assume that $M$ is not a cylinder over a submanifold of $H$. That is to say the helix
constant angle $\theta$ between its tangent spaces and $\d$ is not zero.
We already explained, in the introduction, that the $\eta$-parallel
manifolds $L_{s\eta}$ to $L$ are the slices of $M$.
So by \cite[Theorem 7.1, page 208]{kodai} we get that the $\eta$-parallel
manifolds $L_{s\eta}$ are minimal submanifolds.
\end{proof}

\begin{lema}
\label{traceofshape}
Under the above assumptions, the trace of the shape operator $A_\eta^s$ of $L_{s\eta}$
in direction $\eta$ is given by
\[ \Tr(A_\eta^s)= \Tr((\Di - s \Di^2 -s\No)[\mathbf{1} - 2s \Di + s^2 (\Di^2 + \No)]^{-1}),\]
where $\Di, \No$ are the matrices:  $\Di_{ij}= \lambda_i \delta_{ij}$ and
$\No_{ij}=\langle \nabla^\perp_{E_i} \eta, \nabla^\perp_{E_j} \eta  \rangle$.
\end{lema}
\begin{proof}
As explained in the introduction $\eta=T$ is orthogonal to the slices $L_{s\eta}$.
We will denote by $A_\eta$ the shape operator of $L$ in direction $\eta$.
Let $E_1, \cdots, E_{\mathrm{dim}(L)}$ be the frame of $L$ and let $X_1^s, \cdots, X^s_{\mathrm{dim}(L)}$ be the frame of $L_{s\eta}$ introduced in Lemma \ref{lema:tangent-and-normalframes}. The following computation follows the same ideas as in the classical ``tube formula" (cf. \cite[page 121]{Olmos}):
$$
\begin{aligned}
\langle A_\eta^s(X_i^s), X_j^s \rangle &=- \langle D_{E_i} \eta, X_j^s \rangle =  \langle  \eta, D_{E_i} X_j^s \rangle \\
&= \langle  \eta, D_{E_i} ((1-s \lambda_j) E_j + s \nabla^\perp_{E_j} \eta) \rangle \\
&= (1-s \lambda_j) \langle \eta , \alpha(E_i, E_j) \rangle +
s \langle \eta, \nabla^\perp_{E_i}\nabla^\perp_{E_j} \eta  \rangle\\
&= (1-s \lambda_j) \langle A_\eta(E_i), E_j \rangle -
s \langle \nabla^\perp_{E_i} \eta, \nabla^\perp_{E_j} \eta  \rangle\\
&= (1-s \lambda_j) \lambda_i \delta_{ij}
- s \langle \nabla^\perp_{E_i} \eta, \nabla^\perp_{E_j} \eta  \rangle\\
&= \lambda_i \delta_{ij} - s \lambda_i \lambda_j \delta{ij} - s \langle \nabla^\perp_{E_i} \eta, \nabla^\perp_{E_j} \eta  \rangle
\end{aligned}
$$
Therefore, we have that
$$\langle A_\eta^s(X_i^s), X_j^s \rangle = \Di_{ij} - s \Di_{ij}^2 -s\No_{ij}.$$

Now equation (\ref{metric}) in Lemma \ref{lema:tangent-and-normalframes} give us the metric $G_{ij}$ of $L_{s\eta}$ with respect to the frame $X^s_1, \cdots, X^s_{\mathrm{dim}(L)}$:
\[G_{ij}=\delta_{ij} - s \delta_{ij} (\lambda_i + \lambda_j) + s^2 \lambda_i \lambda_j \delta_{ij} +
s^2 \langle \nabla^\perp_{E_i} \eta, \nabla^\perp_{E_j} \eta  \rangle
\]
So, $$G = \mathbf{1} - 2s \Di + s^2 (\Di^2 + \No).$$
Then, we have that
$$\Tr(A_\eta^s)= \Tr((\Di - s \Di^2 -s\No)[\mathbf{1} - 2s \Di + s^2 (\Di^2 + \No)]^{-1}).$$

\end{proof}

For the proof of Theorem \ref{maintheorem} we will need the following lemma.

\begin{lema}
\label{LA}
Let $\No,\Di$ be symmetric square matrices with $\No$ positive semi-definite. Set $\H := \Di^2 +
\No$ and
let $\epsilon > 0$ be such that the matrix $\mathbf{1} - 2 s \Di + s^2 \H$
is invertible for all $s \in (0,\epsilon)$.
If
\[ \Tr\left((\Di - s \H)(\mathbf{1} - 2 s \Di + s^2 \H)^{-1}\right) = 0 \]
for all $s \in (0,\epsilon)$ then \[ \Di=\No =\H = \mathbf{0} . \]
\end{lema}
\it Proof. \rm The inverse $G^{-1}$ of an invertible matrix $G$ can be
computed by means of its \emph{adjoint} matrix $\adj(G)$. Namely,
\[ G^{-1} = \frac{\adj(G)}{\mathrm{det}(G)} \, .\]

Then for $s \in (0,\epsilon)$ we have
$$
\begin{array}{ccc}
 \Tr\left((\Di - s \H)(\mathbf{1} - 2 s \Di + s^2 \H)^{-1}\right) &=&
\Tr\left((\Di - s \H)\frac{\adj(\mathbf{1} - 2 s \Di + s^2 \H)}{\det \left(
\mathbf{1} - 2 s \Di + s^2 \H \right)} \right) \\
 &=&  0 \, .
\end{array}
$$
Since the polynomial $P(s) := \det \left( \mathbf{1} - 2 s \Di + s^2
\H\right)$ has a finite number of zeros we get that
\[ \Tr\left((\Di - s \H)(\mathbf{1} - 2 s \Di + s^2 \H)^{-1}\right) = 0 \]
for all real numbers $s \in \mathbb{R}$ up to the finite number of zeroes
of $P(s)$.
Changing $s = \frac{1}{t}$ we get
\begin{equation}
\label{traza}
\Tr\left((t\Di - \H)(t^2\mathbf{1} - 2 t \Di + \H)^{-1}\right) = 0
\end{equation}
for all $t \in \mathbb{R}$ up to a finite number of exceptions.\\

Let $\overrightarrow{v} \in \mathrm{ker}(\H)$ be a vector in the kernel of
$\H$ then \[ \H \overrightarrow{v} = \Di^2 \overrightarrow{v} + \No
\overrightarrow{v} = 0 \, \, .\]
So \[ \Di \overrightarrow{v} . \Di \overrightarrow{v} = - \No
\overrightarrow{v} . \overrightarrow{v} \]
hence $\Di \overrightarrow{v} = \No \overrightarrow{v} = \H
\overrightarrow{v} = 0 $ since $\No$ is positive semi-definite.
Then $\mathrm{ker}(\H) \subset \mathrm{ker}(\Di)$ and $\mathrm{ker}(\H) \subset \mathrm{ker}(\No)$.
Since $\Di$ and $\No$ are symmetric matrices they preserve $\mathrm{ker}(\H)^{\perp}$
and we get the following block decomposition with respect to the splitting $\mathrm{ker}(\H) \oplus \mathrm{ker}(\H)^{\perp}$:

\[\Di = \left(
           \begin{array}{cc}
             0 & 0 \\
             0 & \Di_1 \\
           \end{array}
         \right) \, \, , \, \, \No = \left(
           \begin{array}{cc}
             0 & 0 \\
             0 & \No_1 \\
           \end{array}
         \right) \, \, , \, \, \mbox{and} \, \, \H = \left(
           \begin{array}{cc}
             0 & 0 \\
             0 & \H_1 \\
           \end{array}
         \right) \, .
  \]

Now equation (\ref{traza}) reduce to \[ \Tr\left((t\Di_1 -
\H_1)(t^2\mathbf{1} - 2 t \Di_1 + \H_1)^{-1}\right) = 0 \, .\]
Letting  $t \rightarrow 0 $ we get   \[ \Tr\left((- \H_1)(\H_1)^{-1}\right)
= \Tr\left(- \mathbf{1} \right) = 0 \, \]
which is a contradiction unless $\H_1 = 0$. So $\H_1 = 0$ hence $\H = 0$ and also $\Di = \No = 0$ since $\mathrm{ker}(\H)^{\perp} = \{ 0\}$ . $\Box$

\subsection{Proof of Theorem \ref{maintheorem}}

\vspace{0.1cm}

Let $M$ be a ruled minimal helix submanifold of $\mathbb{R}^n$ with constant angle $\theta \neq 0$. We are going to show
that $M$ is not full, that is to say $M$ is contained in a hyperplane.\\
By Proposition \ref{decomposition}, the helix $M$ is a union of
parallel submanifolds $L_{s\eta}$, where $L$ is a slice and $\eta= T$ is a normal
vector field of $L$ of constant length.\\
By Lemma \ref{traceofshape} and since $L_{s\eta}$ is minimal for small values of $s$,
$$0=\Tr(A_\eta^s)= \Tr((\Di - s \Di^2 -s\No)[\mathbf{1} - 2s \Di + s^2 (\Di^2 + \No)]^{-1}),$$
where $\Di, \No$ are the matrices:  $\Di_{ij}= \lambda_i \delta_{ij}$ and
$\No_{ij}=\langle \nabla^\perp_{E_i} \eta, \nabla^\perp_{E_j} \eta  \rangle$.\\
Now, by Lemma \ref{LA}, $\Di=0$ and $\No=0$, that is to say the vector field $\eta$
is parallel with respect to the normal connection and its shape operator $A_\eta=0$.
Hence $\eta$ is constant in the ambient space along $L$.
This implies that $T$ is a constant vector along $M$ in the ambient space $\R^n$ hence $\d^{\top}$ is constant along $M$ in the ambient space $\R^n$.
Therefore, $\d^{\perp}$ is a constant vector field along $M$ in the ambient space $\R^n$.
Since we assumed that $\theta \neq 0$ we get that $M$ is contained in a hyperplane orthogonal to $\d^{\perp} \neq 0$ i.e. $M$ is not full. $\Box$

\subsection{Proof of Corollary \ref{offsets}}

The corollary follows by applying Lemma \ref{lema:tangent-and-normalframes}, Lemma \ref{traceofshape} and Lemma \ref{LA} to $L = M$ and $\eta = \xi$.

\section{Complex helix submanifolds: Proof of Theorem \ref{complex}.}
It is well-known that a complex submanifolds of $\mathbb{C}^n$ is also a
minimal submanifold.
We notice that Theorem \ref{complex} is not an immediate corollary of
Theorem \ref{maintheorem} since we do not assume the complex submanifold $M
\subset \mathbb{C}^m$ to be a {\bf ruled} helix.

We need the following lemma.

\begin{lema} Let $N^2 \subset \mathbb{R}^{n}$ be a minimal helix surface
(not necessarily ruled).
Then $N^2$ is a totally geodesic submanifold (hence ruled).
\end{lema}
\it Proof. \rm Under the hypothesis the induced metric on $N^2$ is flat.
Indeed, this is obvious if the helix angle $\theta$ is zero.
If $\theta \neq 0$ then $N^2$ carries an harmonic eikonal function, hence two
perpendicular totally geodesic foliations, which implies flatness.
Now it is a well-known fact that the Gauss equation implies that a minimal
and Ricci-flat submanifold of $\mathbb{R}^n$ is totally geodesic. $\Box$ \\

\def\J{{\mathrm{J}}}
\def\T{{\mathrm{T}}}

\it Proof of Theorem \ref{complex}.\rm \hspace{.2cm}
We will show that $M^m \subset \mathbb{C}^{n}$ is a ruled helix
submanifold. Let $\d = \cos(\theta) \T + \sin(\theta) \xi $ be the
decomposition of $\d$ in its tangent and normal components. Let $\J$ be the
complex structure of $\mathbb{C}^n$ regarded as an automorphism of
$\mathbb{C}^n$. Then $M$ is also a helix with respect to the direction $\J
\d$. So both $\T$ and $\J \T$ are geodesic vector fields of $M^m$.
Let $\mathcal{T} = \mathrm{span}\{\T,\J \T\}$ be the $2$-dimensional
distribution generated by $\T$ and $\J \T$. We claim that $\mathcal{T}$ is
involutive. Indeed, by computing the bracket we have  \[ \begin{aligned} \J
[\T , \J \T] &= \J \left( \nabla_{\T} \J \T - \nabla_{\J \T} \T \right) \\
&= \J\nabla_{\T} \J \T - \J \nabla_{\J \T} \T \\
&= -\nabla_{\T}\T -  \nabla_{\J \T} \J \T \\
&= 0 - 0
\end{aligned}\]
and so $[\T , \J \T] = 0$ showing that $\mathcal{T}$ is involutive. Notice
that the leaves of $\mathcal{T}$ are complex surfaces which are helix with
respect to both $\d$ and $\J \d$. Then by the above lemma it follows that
the leaves of $\mathcal{T}$ are complex totally geodesic surfaces of
$\mathbb{C}^n$. Therefore the flow lines of both vector fields $\T$
and $\J \T$ are straight lines of $\mathbb{C}^{n}$.
So $M$ is a minimal ruled helix and we can apply Theorem \ref{maintheorem}
to get that $M$ splits as required. $\Box$

Now, we will extend Theorem \ref{complex} to the case when the isometric immersion
of a K\"ahler manifold is not necessarily a holomorphic isometric immersion.
The next statement was taken from \cite{DS} but it is a result of Dajczer and Gromoll.

\begin{teor}{\em (\cite{DG})}
\label{teo:dajczer-gromoll}
Let $M$ be a simply connected K\"ahler manifold (not necessarily complete) and let $f: M \longrightarrow \R^n$ be a minimal
isometric immersion. Then there exists a minimal isometric immersion $g: M \longrightarrow \R^n$
such that $\overline{f}: M \longrightarrow  \R^n \times \R^n = \mathbb{C}^n $ given by
$\overline{f}(p)=(\frac{f(p)}{\sqrt{2}},\frac{g(p)}{\sqrt{2}})$ is isometric and holomorphic
with respect to the complex structure $J(u,v)=(-v,u)$ on $\R^n \times \R^n$.
\end{teor}

We are ready to give the extension of Theorem \ref{complex}.

\begin{coro}
Let $M^m$ be a simply connected K\"ahler manifold (not necessarily complete) and let $f: M \longrightarrow \R^n$
be a minimal isometric immersion.
If under this immersion $M$ is a helix submanifold then $M$ is a cylinder.
\end{coro}
\begin{proof}
We can assume that $f(M)$ is a helix submanifold with respect to the direction induced by
the factor $\R$ in $\R^n= \R \times \R^{n-1}$.
Let us observe that in Theorem \ref{teo:dajczer-gromoll}, we are identifying $\mathbb{C}^n$ with $\R^n \times \R^n$ with the map
$I : \R^n \times \R^n \longrightarrow \mathbb{C}^n$ given by
$$(x_1 ,  x_2 , \cdots  , x_n , y_1 ,  y_2 , \cdots  , y_n ) \mapsto (x_1+ i y_1 , x_2+ i y_2, \cdots , x_n+ i y_n ) .$$
By Theorem \ref{teo:dajczer-gromoll}, $I \circ \overline{f} : M \longrightarrow \mathbb{C}^n$ is
a holomorphic isometric immersion, i.e. $M$ is a K\"ahler submanifold of $\mathbb{C}^n$.
Therefore, Theorem \ref{complex}, implies that
$I \circ \overline{f}(M)= \left\{ I (\frac{f(p)}{\sqrt{2}},\frac{g(p)}{\sqrt{2}})  | p \in M \right\}$
is an extrinsic product
$\mathbb{C} \times N \subset \mathbb{C} \times \mathbb{C}^{n-1}$.
This proves that the original immersed submanifold $f(M)$ is an extrinsic product in $\R \times \R^{n-1}$,
i.e. it is a cylinder.\\
\end{proof}

\section{Minimal Riemannian foliations: Proof of Theorem \ref{MinimalRiemannianFoliation}.}

Let $p \in U$ and let $F_p$ be the leave of $\mathcal{F}$ through $p$.
Let $m = \mathrm{dim}(F_p)$ be the dimension of $F_p$ and let $f: \mathbb{R}^m \to \mathbb{R}^n$
be a parametrization of $F_p$ near $p$ i.e. $f(0) = p$ and $f(W)$ is an open subset of $F_p$ for a neighborhood $W$ of $0$.

Due to the fact that $\mathcal{F}$ is a Riemannian foliation we have that for
$q \in U$ near to $p$ the leave $F_q$ is obtained from $F_p$ and a normal vector
field $\xi \in \Gamma(\nu(F_p))$ of constant length. Namely, $f_{t\xi}(x) := f(x) + t\xi(x)$ is parametrization of
a neighborhood of $q \in F_{p + t \xi(p)}$ for small fixed $t$.\\
Then Corollary \ref{offsets} implies that $\xi$ is constant in $\mathbb{R}^n$ along $f(W) \subset F_p$.
That is to say $f(W)$ is contained in the affine hyperplane
\[ H_{\xi}:= \{ x \in \mathbb{R}^n :  \langle \xi(p), x \rangle = \langle \xi(p), p \rangle \} \, .\]
Since $\mathcal{F}$ is a foliation of $U$ we get that for each normal direction $\xi \in \nu_p(F_p)$
$f(W)$ is contained in the hyperplane $H_{\xi}$. So $F_p$ is near $p$ an open subset of an affine
subspace and the Riemannian foliation $\mathcal{F}$ consist of the parallel affine subspaces as we wanted to show.

Since complex submanifolds of $\mathbb{C}^n$ are minimal submanifolds the last claim of Theorem \ref{MinimalRiemannianFoliation}
follows from the first part.

\section{The geometry of the helix submanifolds}\label{general}

In this section we investigate some relations between the extrinsic geometry of the the helix
$M$ and the intrinsic geometry of its base $B=\pi(M) \subset \mathbb{R}^n$.
Our analysis is based on the eikonal function of the projection method.
The notation $\alpha_B$ and $ \mathbf{H}_B$ means respectively the second
fundamental form of the submanifold $B \subset \mathbb{R}^{n-1} \subset  \mathbb{R}^n$
 and its mean curvature vector field.
The gradient  $\nabla_B f$ and the Laplacian $\Delta_B f$ of the function $f$ are computed with respect the Riemannian metric on $B$ induced by the inclusion $B \subset \mathbb{R}^n$.

\begin{teor}
\label{formulae}
Let $B$ be the base of the helix $M$ and let $f \in C^{\infty}(B)$ be the associated eikonal function.
Then $M$ is a minimal submanifold of $\mathbb{R}^n$ if and only if the following holds:
\[ \begin{cases}\mathbf{H}_B = \frac{\alpha_B(\nabla_B f, \nabla_B f)}{1 + \|\nabla_B f\|^2} \, , \\
\Delta_B f = 0  \, . \end{cases} \]
\end{teor}
\it Proof. \rm Let $\xi_1,\cdots,\xi_r \in \Gamma(\nu(B))$ be a (local) normal frame of $B \subset \R^{n-1}$.
Then the vectors $\xi_1(p),\cdots, \xi_r(p)$ are also normal to $M$ at the point $\phi(p) = (p,f(p)) \in M$.
The vector field \[ N = \frac{(\nabla_B f, -1)}{\sqrt{1 + \|\nabla_B f\|^2}} \] is normal to $M$ so $N, \xi_1, \cdots, \xi_r$ is a normal frame of $M$.\\
Let $E_1,\cdots, E_{\mathrm{dim}(B)}$ be an orthonormal local frame of $B$ with $E_1:= \frac{\nabla_B f}{\|\nabla_B f\|}$.
Then the vector fields $X_1, \cdots, X_{\mathrm{dim}(M)}$ defined by \[ X_i := (E_i, \dif f(E_i)) \in \R^{n-1} \times \R \] give us a frame of $M$.


In terms of this frame the second fundamental form
$\alpha_M$ of $M$ is given by \[ \langle \alpha_M(X_i,X_j) , \xi_k \rangle =
\langle D_{E_i}E_j + E_i(\dif f(E_j))\d , \xi_k \rangle = \langle \alpha_B(E_i,E_j),\xi_k  \rangle \]
\[ \begin{aligned} \langle \alpha_M(X_i,X_j) , N \rangle &= \langle D_{E_i}E_j + E_i(\dif f(E_j))\d , \frac{(\nabla_B f, -1)}{\sqrt{1 +\|\nabla_B f\|^2}}  \rangle \\
&= \langle \nabla_{E_i}E_j,\frac{\nabla_B f}{\sqrt{1 + \|\nabla_B f\|^2}} \rangle - \frac{E_i(\dif f(E_j))}{\sqrt{1 + \|\nabla_B f\|^2}} \\
&= \langle \nabla_{E_i}E_j,\frac{\nabla_B f}{\sqrt{1 + \|\nabla_B f\|^2}} \rangle
 \end{aligned}\]

Let $G = (G_{i j} = \langle X_i , X_j \rangle)$ be the matrix of the metric of $M$ with respect to the frame $X_1,\cdots,X_{\mathrm{dim}(M)}$. Then the matrix of the shape operators $A_N, A_{\xi}$ with respect to the frame $X_1,\cdots,X_{\mathrm{dim}(M)}$ are:
\[ A_{\xi_k} = A_k G^{-1} \, \, , \, \, A_N = R G^{-1} \]
where $(A_k)_{i,j} := \langle \alpha_M(X_i,X_j) , \xi_k \rangle = \langle \alpha_B(E_i,E_j),\xi_k  \rangle$ and $(R_{i j}) :=  \langle \alpha_M(X_i,X_j) , N \rangle$. Observe that $G$ is the diagonal matrix $G = \mathrm{diag}(1+\|\nabla_Bf\|^2, 1, \cdots,1)$ since $df(E_1)= \| \nabla_B f\|$.

Then for all $\xi_k$ we have
\[ \begin{aligned} \mathrm{trace}(A_{\xi_k}) &= \mathrm{trace}(A_k G^{-1}) = \\
&= \frac{\langle \alpha_B(E_1,E_1),\xi_k  \rangle}{1+\|\nabla_Bf\|^2} + \langle \alpha_B(E_2,E_2),\xi_k  \rangle + \cdots \\
&+ \langle \alpha_B(E_{\mathrm{dim}(B)},E_{\mathrm{dim}(B)}),\xi_k  \rangle   \\
&= \frac{\langle \alpha_B(E_1,E_1),\xi_k  \rangle}{1+\|\nabla_Bf\|^2} - \langle \alpha_B(E_1,E_1),\xi_k  \rangle + \langle \mathbf{H}_B, \xi_k \rangle \, .
\end{aligned}\]
So $\mathrm{trace}(A_{\xi_k})=0$ for all $k$ if and only if   \[ \mathbf{H}_B = \frac{\|\nabla_Bf\|^2}{1+\|\nabla_Bf\|^2 } \alpha_B(E_1,E_1) = \frac{\alpha_B(\nabla_B f, \nabla_B f)}{1 + \|\nabla_B f\|^2} \]
and we get the first identity.
We also have  \[ \begin{aligned} \mathrm{trace}(A_{N}) &= \mathrm{trace}(RG^{-1}) = \\
&= \frac{ \langle \nabla_{E_1}E_1,\frac{\nabla_B f}{\sqrt{1 + \|\nabla_B f\|^2}} \rangle}{1+\|\nabla_Bf\|^2} + \sum_{j=2}^{\mathrm{dim}(B)}\langle \nabla_{E_j}E_j,\frac{\nabla_B f}{\sqrt{1 + \|\nabla_B f\|^2}} \rangle =\\
&= \frac{ \langle \nabla_{E_1}E_1,\frac{\nabla_B f}{\sqrt{1 + \|\nabla_B f\|^2}} \rangle}{1+\|\nabla_Bf\|^2}+\frac{\langle \nabla_{E_1} \nabla_B f, E_1 \rangle}{\sqrt{1 + \|\nabla_B f\|^2}} - \frac{\Delta_B f}{\sqrt{1 + \|\nabla_B f\|^2}}\\
&= - \frac{\Delta_B f}{\sqrt{1 + \|\nabla_B f\|^2}}
\end{aligned}\]
the last equation follows from the fact that $E_1 = \frac{\nabla_B f}{\|\nabla_B f\|}$ and $\|\nabla_B f\|$ is a constant.
So $\mathrm{trace}(A_{N}) = 0$ if and only if $\Delta_B f = 0$.
$\Box$

\vspace{.5cm}

An interesting application of the above result is given in Theorem \ref{baseN} below.

\subsection{The intrinsic geometry of helix submanifolds}

\label{subsec:intrinsic}

As we recall in the introduction any helix submanifold $M$ is locally constructed with the projection
method where we used a Riemannian manifold $B:= \pi(M) \subset \mathbb{R}^{n-1} \subset \mathbb{R}^n $
called the basis.
Here we study the relations between the geometries of $M$ and $B$.\\

So if we want to construct a helix $M$ in $\mathbb{R}^n$,
we can consider a Riemannian manifold $(B, g)$ of dimension $m$
with  an immersion of $(B, g)$ in $\mathbb{R}^n$ given by
$\phi(p)=(i(p), f(p))$ where $i: B \longrightarrow \mathbb{R}^{n-1}$ is an isometric immersion
and where $f:B \longrightarrow \R$ is an non constant eikonal function on $B$.
By Theorem \ref{teor:eikonalvshelix}, $M= \phi(B)$ is a helix submanifold of $\mathbb{R}^n$
with its induced metric $H$.
Then we have an isometry between $(M,H)$ and $(B, h:= \phi^* H)$.
First, let us observe that the relation between the metrics of $(B,g)$ and $(B,  h)$
is given by
$$h(X,Y):=(\phi^* H)(X,Y) =H(\phi_* (X), \phi_* (Y)) = g(X,Y) +df(X) df(Y) .$$
{\bf So, in this subsection we will compare $(B,g)$ with $(B,  h)$ and $f:(B,g) \longrightarrow \R$
will be a non constant $C^\infty$ eikonal function.}\\
Let $E_1 = \frac{\nabla_g f}{\|\nabla_g f\|}, E_2, \cdots , E_m$ be a local frame orthonormal of $(B,g)$.
Since $h(E_1, E_1) = 1 + \|\nabla_g f\|^2$, we can consider the following orthonormal local frame
of $(B,h)$: $ \tilde{E}_1 = \frac{1}{ \sqrt{1 + \|\nabla_g f\|^2}} E_1, E_2 , \cdots , E_m$.\\
Let us observe that in the basis $E_1 = \frac{\nabla_g f}{\|\nabla_g f\|}, E_2, \cdots , E_m$, the relation
between the metrics looks like
\begin{equation}
\label{eqn:h&ginbaseE_i}
h(E_i, E_j)=
\left\{
\begin{array}{c}
g(E_i, E_j)= \delta_{ij}, \mbox{ if either } i > 1 \mbox{ or } j > 1,\\
(1 + \|\nabla_g f\|^2) g(E_1,E_1), \mbox{ if } i = j = 1.
\end{array}  \right.
\end{equation}

\begin{rem}
\label{rema:identificaciones}
\em
Under $\phi$ the local vector field $\tilde{E_1}$ is identified with $T= \d^{\top}/\|\d^{\top}\|$ the
unit tangent component of the helix direction $\d$. Indeed,
\[ \begin{aligned} \phi_*(\tilde{E_1}) &= \frac{1}{\|\nabla_g f\| \sqrt{1 + \|\nabla_g f\|^2}} \phi_*(\nabla_g f) \\
&= \frac{1}{\|\nabla_g f\| \sqrt{1 + \|\nabla_g f\|^2}} (\nabla_g f + \|\nabla_g f\|^2 \d) = T \, . \end{aligned} \]
Notice that the function $f$ regarded as a function of $M$ is given by the
height function $f(x)= \langle x , \d \rangle $ with $x \in M$. So the gradient in $M$ of $f$ is $\d^{\top}$
and the unitary projection $\eta$ of $\d^{\top}$ in $B$ is a constant multiple of the gradient of $\nabla_g f$ when we regard $f$ as a function of $B$.
\end{rem}

In the next Proposition \ref{prop:volume-forms}, we give the relation between the volume forms of the
metrics $h$ and $g$.

\begin{prop}
\label{prop:volume-forms}
Let $\omega_g$ and $\omega_h$ be the volume forms of $(B,g)$ and $(B,h)$, respectively.
Then $$\omega_h = \sqrt{1+ \|\nabla_g f\|^2}  \  \omega_g.$$
\end{prop}
\begin{proof}
Let $E_1 = \frac{\nabla_g f}{\|\nabla_g f\|}, E_2, \cdots , E_m$ be the basis defined above.
The volume forms are given by
$\omega_g(E_1, \cdots , E_m)=\sqrt{\det (g(  E_i, E_j  ))} = 1$ because the basis is orthonormal with the metric $g$.
In the case of metric $h$ we have:
$\omega_h(E_1, \cdots , E_m)=\sqrt{\det ( h(  E_i, E_j  ))} = \sqrt{1+ \|\nabla_g f\|^2}$.
\end{proof}

\begin{prop}
\label{prop:gradients}
Let $\nabla_g f$ and $\nabla_h f$ be the gradients of $f$ in $(B,g)$ and $(B,h)$, respectively.
Then
\begin{equation}
\label{eqn:gradientes}
\nabla_h f = \frac{1}{1 + \|\nabla_g f\|^2 } \nabla_g f.
\end{equation}
\end{prop}
\begin{proof}
For every $j$, we have the relation:
$$h(\nabla_h f, E_j) = df(E_j) = g(\nabla_g f, E_j) $$
and in particular we have for $j>2$:\\
$h(\nabla_h f, E_j) = g(\nabla_g f, E_j)=0$. When $j=1$:
$E_1 = \frac{\nabla_g f}{\|\nabla_g f\|}$,
$$h(\nabla_h f, E_1) =  g(\nabla_g f, E_1) = g(\nabla_g f, \frac{\nabla_g f}{\|\nabla_g f\|}) =\|\nabla_g f\| .$$
We can calculate $\nabla_h f$ as
\begin{eqnarray*}
\nabla_h f &=& \frac{1}{ 1 + \|\nabla_g f\|^2} h(\nabla_h f, E_1)E_1   =  \frac{1}{1 + \|\nabla_g f\|^2 } \nabla_g f.
\end{eqnarray*}
\end{proof}

\begin{prop}
\label{prop:conexiones}
Let $\nabla_g f$ be the gradient of $f$ in $(B,g)$. Then the Levi-Civita connection $\nabla^h$
of $(B,h)$ is given by
\begin{equation}
\label{eqn:conexiones}
\nabla^h_X Y = \nabla^g_X Y +  \frac{Hess_g f(X,Y)}{1 +\|\nabla_g f\|^2} \nabla_g f.
\end{equation}
\end{prop}
\begin{proof}
Let us recall Koszul's formula:\\
\begin{eqnarray*}
 2 g(\nabla^g_X Y,Z) & = & Xg(Y,Z) - Zg(X,Y) + Yg(Z,X) \\
& - & g(X, [Y,Z]) + g(Z,[X,Y]) + g(Y,[Z,X]).
\end{eqnarray*}
To prove the relation (\ref{eqn:conexiones}), we only have to check it for $X$ and $Y$ in a local frame.
Let $E_1 = \frac{\nabla_g f}{\|\nabla_g f\|}, E_2 \cdots , E_m$ be a local frame orthonormal of $(B,g)$.
Since $h(E_1, E_1) = 1 + \|\nabla_g f\|^2$, we can consider the following orthonormal local frame
of $(B,h)$: $ \tilde{E}_1 = \frac{1}{ \sqrt{1 + \|\nabla_g f\|^2}} E_1, E_2 , \cdots , E_m$.\\
Using Koszul's formula: $i, \ j , \ k > 1$,
\begin{eqnarray*}
2g( \nabla^g_{E_j} E_i, E_1) &=& - g(E_j,[E_i, E_1] )+ g(E_1,[E_j, E_i] ) + g(E_i, [E_1, E_j] )\\
& = & - g(E_j,[E_i, E_1]) + g(E_i,[E_1, E_j]).\\
2g( \nabla^g_{E_j} E_i, E_k) &=& - g(E_j,[E_i, E_k]) + g( E_k,[E_j, E_i]) + g(E_i,[E_k, E_j]).\\
2g( \nabla^g_{E_1} E_i, E_k) &=& - g(E_1,[E_i, E_k] ) + g(E_k, [E_1, E_i] ) + g(E_i,[E_k, E_1] ),\\
& = &  g(E_k,[E_1, E_i])  + g(E_i ,[E_k, E_1] )  .
\end{eqnarray*}
A similar calculus and the properties\\
$h( E_1, [E_i, E_j]  )=0$, $h(E_j,[E_i, E_k]) = g(E_j,[E_i, E_k])$, $h(E_j,[E_i, E_1]) = g(E_j,[E_i, E_1])$
(see (\ref{eqn:h&ginbaseE_i}))
proves that:
\begin{eqnarray*}
h( \nabla^h_{E_j} E_i, E_1) &=& g( \nabla^g_{E_j} E_i, E_1)  , \\
h( \nabla^h_{E_j} E_i, E_k) &=&  g( \nabla^g_{E_j} E_i, E_k)  , \\
h( \nabla^h_{E_1} E_i, E_k) &=&  g( \nabla^g_{E_1} E_i, E_k)  .
\end{eqnarray*}
Thus we can calculate for $i, \ j > 1$,
\begin{eqnarray*}
\nabla^h_{E_j} E_i &=& h(\nabla^h_{E_j} E_i, \tilde{E}_1) \tilde{E}_1 + \sum_{k>1} h(\nabla^h_{E_j} E_i, E_k) E_k \\
 & = & \frac{1}{1 + \|\nabla_g f\|^2} h(\nabla^h_{E_j} E_i, E_1) E_1 + \sum_{k>1} h(\nabla^h_{E_j} E_i, E_k) E_k\\
 & = &  \nabla^g_{E_j} E_i -\frac{\|\nabla_g f\|^2}{1 + \|\nabla_g f\|^2}  g(\nabla^g_{E_j} E_i, E_1) E_1 .
\end{eqnarray*}
Let us analyse the last term:
\begin{eqnarray*}
-g(\nabla^g_{E_j} E_i, E_1) & = &  g( E_i, \nabla^g_{E_j} E_1 )  = \frac{1}{\|\nabla_g f\|} g( E_i, \nabla^g_{E_j} (\nabla_g  f) )\\
& = & \frac{1}{\|\nabla_g f\|} Hess_g f (E_i , E_j ).\\
\mbox{Therefore},\\
\nabla^h_{E_j} E_i &=& \nabla^g_{E_j} E_i +\frac{\|\nabla_g f\|}{1 + \|\nabla_g f\|^2} Hess_g f (E_i , E_j )  E_1 \\
 & = & \nabla^g_{E_j} E_i +\frac{ Hess_g f (E_i , E_j )}{1 + \|\nabla_g f\|^2} \nabla_g f.
\end{eqnarray*}
When $i=1$ or $j=1$, $\nabla^h_{E_j} E_i =  \nabla^g_{E_j} E_i $.
Since $f$ is eikonal in $(B,g)$ and by Proposition \ref{prop:gradients}, we deduce that
$f$ is eikonal in $(B,h)$. Therefore, $\nabla^h_{E_1} E_1 = \nabla^g_{E_1} E_1 =0 $.
Finally, other consequence is that for every $X \in TB$, $Hess_g f (E_1 , X )=0$.
\end{proof}

\begin{rema}
\em
Let us observe that Equations (\ref{eqn:gradientes}) and (\ref{eqn:conexiones}) implies that if
$\nabla_g f$ is parallel in $(B,g)$ then $\nabla_h f$ is a parallel vector field in
$(B,h)$. Also it is true that the integral lines of $\nabla_g f$ are geodesics in $(B,g)$
if and only if the integral lines of $\nabla_h f$ are geodesics in $(B,h)$, i.e.
$ \nabla^h_{\nabla_h f} \nabla_h f =0 \mbox{ if and only if } \nabla^g_{\nabla_g f} \nabla_g f =0 $.
\end{rema}

\begin{prop}
\label{prop:hessians}
Let $\nabla_g f$ and $Hess_g f$ be the gradient and the Hessian repectively, of $f$ in $(B,g)$.
Then
\begin{equation}
\label{eqn:formula-hessianos}
Hess_h f = \frac{1}{1 + \|\nabla_g f\|^2 } Hess_g f.
\end{equation}
\end{prop}
\begin{proof}
If $i , \ j > 1 $ we have that,
\begin{eqnarray*}
& &Hess_h f (E_i, E_j) = \\
 & = &  h(\nabla^h_{E_i} (\nabla_h f) , Ej)=  \frac{1}{1 + \|\nabla_g f\|^2 } h(\nabla^h_{E_i} (\nabla_g f) , Ej)   \\
& = & \frac{\|\nabla_g f\|}{1 + \|\nabla_g f\|^2 } h(\nabla^h_{E_i} E_1 , Ej)
 = - \frac{\|\nabla_g f\|}{1 + \|\nabla_g f\|^2 } h( E_1 , \nabla^h_{E_i} Ej)\\
&=& - \frac{\|\nabla_g f\|}{1 + \|\nabla_g f\|^2 } g( E_1 , \nabla^g_{E_i} Ej)=
\frac{\|\nabla_g f\|}{1 + \|\nabla_g f\|^2 } g( \nabla^g_{E_i} E_1 , Ej)\\
&  = & \frac{1}{1 + \|\nabla_g f\|^2  } Hess_g f (E_i, E_j).
\end{eqnarray*}
Finally,
\begin{eqnarray}
\label{eqn:h} Hess_h f (E_1, E_j) = h( \nabla^h_{E_1} (\nabla_h f), E_j) &=& 0 \\
\label{eqn:g} Hess_g f (E_1, E_j) = g( \nabla^g_{E_1} (\nabla_g f), E_j) &=& 0 ,
\end{eqnarray}
because $\nabla^h_{E_1} E_1 = 0$ , $\nabla^g_{E_1} E_1 = 0$.
The property that $f$ is eikonal both in $(B,g)$ and $(B,h)$
implies the latter two equalities.
\end{proof}

\begin{coro}
\em
\label{coro:laplacian}
\label{coro:relacion-laplacians}
The relation between the Laplacians is given by
$$ \triangle_h f = \frac{1}{1 + \|\nabla_g f\|^2 } \triangle_g f, $$
where $\triangle_h f$ and $\triangle_g f$ are the Laplacians of $f$ in
$(B,h)$ and $(B,g)$, respectively.
\end{coro}
\begin{proof}
It follows by taking the trace in both sides of formula (\ref{eqn:formula-hessianos})
and applying (\ref{eqn:h}) and (\ref{eqn:g}).
\end{proof}

As an application we obtain a different proof of the second part of Theorem \ref{formulae}.

Let us observe that we have applied two notations $\triangle_g f$ and $\triangle_B f$
which are the same: The Laplacian for the isometric immersion of $(B,g)$ in
$\mathbb{R}^{n-1} \subset \mathbb{R}^n$ where $B=\pi(M)$ is the projection of the helix $M$.
Moreover, the metric of the helix $M$ is $(M,H)$ wich is isometric to $(B,h)$.

\begin{coro}
Let $M$ be a helix submanifold. Let $f$ be the associated eikonal function $f:B=\pi(M) \rightarrow \R$.
If $M$ is minimal then $\triangle_B f = 0$, in particular $f$ is an isoparametric function.
\end{coro}
\begin{proof}
Since $M$ is a helix submanifold, locally $M =\{ (x, f(x))  \}$ where $f : B \longrightarrow \R$ is a height function.
It is well known that the height functions of $M$
are harmonic with the metric of $M$ because $M$ is minimal.
Therefore $\triangle_h f =0$.
Therefore by Corollary \ref{coro:relacion-laplacians},
$ \triangle_g f = \triangle_h f =0$.
So, $\triangle_B f  = \triangle_g f = 0$.
\end{proof}

\begin{rema}
\em
Let us recall that a height function on $M$, $f:M \longrightarrow \R$ given by
$f(x) = \langle x, \d \rangle$ is harmonic when
the submanifold is minimal. Here $\d$ is a unit direction in $\R^n$.
In our case of helix submanifolds, there is other way to calculate the Laplacian of a height function:\\
According to \cite[page 194]{kodai} for any helix submanifold we have the
structure equation
$$\nabla_X T = \tan(\theta) A^\xi (X)$$ with $A^\xi$ the shape operator of the immersion $M \subset \R^n$
with respect to the vector $\xi = \d^\perp/\|\d^\perp\|$.
Taking an orthonormal basis of $TM$ we can do the sum over the basis to obtain that
$$\triangle_M f = \sum_{i=1}^m \langle \nabla_{X_i} (\nabla_M f) , X_i \rangle =
\cos(\theta) \sum_{i=1}^m \langle \nabla_{X_i} T , X_i \rangle
= \sin(\theta) \langle H , \xi \rangle ,$$
where $ \cos(\theta) = \langle T, \d \rangle = \|\d^{\top}\|, \ \sin(\theta) = \langle \xi , \d  \rangle $, $\d^{\top} = \nabla_M  f $
and $T= \d^{\top} / \|\d^{\top}\| = \nabla_M f / \cos(\theta)$.
So, it is clear that if $M$ is minimal then $f$ is harmonic in $M$.
In general, it is well known the formula for Euclidean immersed submanifolds
$\triangle_M f = \langle H, \d \rangle $.
The two relations for the Laplacian are compatible because
$\d = \cos(\theta) T + \sin(\theta) \xi $.
\end{rema}

Now we are going to find a relation between the Ricci curvature $Ric_g$ of $(B,g)$ and $Ric_h$ of $(B,h)$.\\
The Riemannian tensor of curvature is given by
$$R(X,Y)Z = - \nabla_X \nabla_Y Z + \nabla_Y \nabla_X Z + \nabla_{[X,Y]} Z , $$
and the Ricci curvature
$$Ric(X,Y) = \sum_{i=1}^m \langle R(X, X_j) Y , X_j  \rangle  , $$
where $X_1, \ldots , X_m$ is an orthonormal basis of $TB$.

\begin{prop}
The Ricci curvature $Ric_h$ in direction $\nabla_h f$ is related to the Ricci curvature
$Ric_g$ in direction $\nabla_g f$ by the formula
\begin{equation}
\label{eqn:ricci-relation}
Ric_h(\nabla_h f , \nabla_h f) = \frac{1}{(1 + \|\nabla_g f\|^2)^2} Ric_g(\nabla_g f, \nabla_g f)  .
\end{equation}
\end{prop}
\begin{proof}
Let $E_1= \frac{\nabla_g f}{\|\nabla_g f\|}, \ldots , E_m$ and $\tilde{E}_1 = \frac{1}{ \sqrt{1 + \|\nabla_g f\|^2}} E_1, E_2 \ldots , E_m$
be the local orthonormal frames defined in the beginning of Subsection \ref{subsec:intrinsic}.
Let us observe that for every $Y \in TB$,
\[ Hess_g f (\nabla_g f, Y) = \langle \nabla^g_{\nabla_g f} \nabla_g f, Y \rangle = 0 \]
because the integral lines of $\nabla_g f$ are geodesics of $(B,g)$.
It follows from formula (\ref{eqn:conexiones}) that for every $X \in TB$,
$$\nabla^h_X (\nabla_g f) = \nabla^g_X (\nabla_g f) , \ \  \nabla^h_{\nabla_g f} X = \nabla^g_{\nabla_g f} X . $$
We deduce by substitution that
$$ \nabla^h_{\nabla_g f} \nabla^h_Y (\nabla_g f) = \nabla^h_{\nabla_g f} \nabla^g_Y (\nabla_g f) =
\nabla^g_{\nabla_g f} \nabla^g_Y (\nabla_g f) . $$
Analogously,
$$ \nabla^h_{[\nabla_g f , Y]} (\nabla_g f) = \nabla^g_{[\nabla_g f , Y]} (\nabla_g f) .$$
Since the integral curves of $\nabla_h f$ and $\nabla_g f$ are geodesics in $(B,g)$ and $(B,h)$ respectively,
$$\nabla^h_Y \nabla^h_{\nabla_g f} (\nabla_g f) = 0 =  \nabla^g_Y \nabla^g_{\nabla_g f} (\nabla_g f) . $$
By definition,\\
$R^g(\nabla_g f,Y) \nabla_g f = - \nabla^g_{\nabla_g f} \nabla^g_Y \nabla_g f +
\nabla^g_Y \nabla_{\nabla_g f} \nabla_g f + \nabla^g_{[\nabla_g f,Y]} \nabla_g f $ and a similarly formula
for $R^h$. Then
$$ R^h(\nabla_g f,Y) \nabla_g f = R^g(\nabla_g f,Y) \nabla_g f  .$$
Therefore,
$$
\begin{array}{ccc}
Ric^h(\nabla_g f,\nabla_g f) &=& h( R^h(\nabla_g f, \tilde{E}_1) \nabla_g f , \tilde{E}_1 )  \\
&+& \sum_{i=2}^m h( R^h(\nabla_g f, E_j) \nabla_g f , E_j )\\
&=& \sum_{i=2}^m h( R^h(\nabla_g f, E_j) \nabla_g f , E_j ) \\
&=& \sum_{i=2}^m g( R^g(\nabla_g f, E_j) \nabla_g f , E_j ) = Ric^g(\nabla_g f,\nabla_g f) .\\
\end{array}
$$
To obtain formula (\ref{eqn:ricci-relation}), we have to use equation (\ref{eqn:gradientes})
which is the relation between the gradients $\nabla_g f$ and $\nabla_h f$.
\end{proof}

\begin{coro}
Let $M$ be an immersed helix hypersurface in $\R^{n+1}$ with respect to an unitary direction $\d \in \R^{n+1}$.
Then the Ricci curvature of $M$ in direction of the tangent component of $\d$ is zero:
$$Ricc_M(T,T)=0 , $$
where $T= \d^{\top}/\|d^{\top}\|$.
\end{coro}
\begin{proof}
If $M$ is a cylinder with direction $d$, we are ready.
Otherwise, let $B$ be as before: The orthogonal projection of $M$ into an open part of a hyperplane
$span \{ d^\perp \} $ orthogonal to $d$.
So, $B$ is Ricci-flat because it is an open part of a Euclidean space
and in particular $Ricc_g(\eta,\eta)=0$, where $\eta$ is the unitary projection of
$T$ into the hyperplane $span \{ d^\perp \}$. Let us observe that $\eta$ is a constant multiple of
$\nabla_g f$: By Remark \ref{rema:identificaciones},
$T=\frac{1}{\|\nabla_g f\| \sqrt{1 + \|\nabla_g f\|^2}} (\nabla_g f + \|\nabla_g f\|^2 \d) $
and so $\eta $ is a constant multiple of
$T -\frac{1}{\|\nabla_g f\| \sqrt{1 + \|\nabla_g f\|^2}}  \|\nabla_g f\|^2 \d =
\frac{1}{\|\nabla_g f\| \sqrt{1 + \|\nabla_g f\|^2}} \nabla_g f $. In fact since we are looking for
$\eta$ to be unitary in $(B,g)$ we deduce that $\eta= E_1 = \nabla_g f / \|\nabla_g f\|$.
Since $(M,H)$ and $(B,h)$ are isometric, $Ric_M(T,T) = Ric_h(\tilde{E_1},\tilde{E_1})$.
By Equation (\ref{eqn:ricci-relation}),  $Ric_h(\tilde{E_1},\tilde{E_1})$ is a constant multiple of $Ricc_g(\eta,\eta)$, see
Remark \ref{rema:identificaciones}.
This relations prove that $Ric_M(T,T)=0$.
\end{proof}
\begin{rem}
\em
Another proof of the above corollary is as follows. Notice that if $M$ is helix hypersurface then the vector field $T$ is in the relative nullity distribution i.e. the kernel of the shape operator.
So by Gauss equation the curvature tensor of $M$ vanish when contracted with $T$ hence $Ricc_M(T,T) = 0$.
\end{rem}

\begin{ejem}
\em
\label{ejem:funcion-armonica-eikonal}
Let us consider the Sol geometry: $(\R^3, g_{Sol})$, where the metric is $g_{Sol}= e^{2z} dx^2 + e^{-2z} dy^2 + dz^2$.
The function $f : \R^3 \longrightarrow \R$ given by $f(x,y,z)=z$ is harmonic, see Corollary 4.3 in \cite{Ou}.
This function is also eikonal, its gradient $ \nabla f = \partial_z $ has constant length, it satisfies that $\|\nabla f \|= 1$.
We should remark that the level hypersurfaces are minimal submanifolds but not totally geodesic,
because the latter condition is equivalent to the parallelism of the gradient vector field $\nabla f = \partial_z$.
We can see using the formula of Koszul that this vector field satisfies that $\nabla_{\partial_x} \partial_z = \partial_x$, i.e. $\partial_z$
is not a parallel vector field. Similarly, we have the following relations
\begin{eqnarray*}
\nabla_{\partial_x} \partial_x &=& -e^{2z} \partial_z, \ \nabla_{\partial_x} \partial_y = 0 , \nabla_{\partial_x} \partial_z =  \partial_x\\
\nabla_{\partial_y} \partial_y &=& e^{-2z} \partial_z , \ \nabla_{\partial_y} \partial_z = - \partial_y \\
\nabla_{\partial_z} \partial_z &=& 0 .
\end{eqnarray*}
Now, we are ready for the calculus of the Riemannian curvature tensor, for example
$$
R(\partial_x, \partial_y) \partial_x = e^{2z} \partial_y, \ R(\partial_x, \partial_z) \partial_x = -e^{2z} \partial_z .\\
$$
Therefore,
$$
\langle R(\partial_x, \partial_y) \partial_x , \partial_y \rangle = 1 , \
\langle R(\partial_x, \partial_z) \partial_x , \partial_z \rangle = -e^{2z}  .
$$
Finally, a direct calculus show that the Ricci curvature of this Sol geometry is
\begin{eqnarray*}
Ric(\partial_x, \partial_x)&=&0, \ Ric(\partial_y, \partial_y)= 0, \ Ric(\partial_z, \partial_z)= -2 \\
Ric(\partial_x, \partial_y)&=&0, \ Ric(\partial_x, \partial_z)= 0, \ Ric(\partial_y, \partial_z)= 0.
\end{eqnarray*}
\end{ejem}

From this we conclude that $(\R^3, g_{Sol})$ can not be isometrically immersed
as a minimal submanifold (even locally) in any euclidean space $\mathbb{R}^n$ of any dimension.
Indeed, assume that such isometric immersion do exists. Then from Gauss equation we get
 that the kernel of the Ricci tensor $\mathrm{ker}(Ric) = span \{ \partial_x, \partial_y\}$ is the kernel of the second fundamental form of
 the immersion, i.e. the so called relative nullity distribution. Since this distribution has dimension 2 we get that $(\R^3, g_{Sol})$ is a
 flat Riemannian manifold. This contradicts $Ric(\partial_z, \partial_z)= -2$ and prove our claim.

\subsection{Other results about minimal helices}

Let $M \subset \mathbb{R}^n$ be helix with respect to the direction $\d \in
\mathbb{R}^n$.
Let $\pi: \mathbb{R}^n \rightarrow span \{ \d^{\perp} \} $ be the projection to a normal
hyperplane $\d^{\perp}$ to $\d$.
Since we work locally we can assume that $\pi(M)$ is a submanifold of
$\mathbb{R}^{n-1} \cong span\{ \d^{\perp} \} $.



\begin{teor}{\label{baseN}} Let $M \subset \mathbb{R}^n$ be a full minimal
helix of any codimension with respect to the direction $\d \in \mathbb{R}^n$.
If the Ricci curvature of the submanifold $B:=\pi(M)$ is non-negative then
$M$ is a totally geodesic submanifold of $\mathbb{R}^n$.
\end{teor}

\it Proof. \rm
If $M$ is a cylinder, then $B$ is minimal with non-negative Ricci curvature and therefore
$(B=\pi(M),g)$ is totally geodesic. It is a consequence that a cylinder over a totally geodesic
submanifold is also totally geodesic.
Otherwise, we can apply the projection method
where is important the condition  $\theta \neq 0$.
By \cite[Theorem ]{mona} we have that locally the immersion $M \subset
\mathbb{R}^n$ is given as \[ \phi(p) =  (p, f(p)) \, \]
where $M = \phi(B)$ locally.
Notice that $\phi: (B,h) \rightarrow M \subset  \mathbb{R}^n$ is a isometry.
The function $f$ is eikonal either in $(B=\pi(M) , g) $ or $(B , g)$  and Theorem \ref{formulae} implies
$\Delta_B f := \Delta_g f = 0$. Bochner's formula for functions
together with the hypothesis that $(B,g)$ has non-negative Ricci curvature
implies that the $\nabla_B f:=\nabla_g f$ is a parallel vector field of $(B ,g)$
and therefore, $\nabla_h f$ is parallel in $(B,h)$. Since $\phi$ is a isometry and by
Remark \ref{rema:identificaciones}, $\phi_*(\tilde{E_1}) = T$ we deduce that
$T$ is parallel in $M$.
In particular $\mathrm{Ric}_B (\nabla_g f) = 0$.
By using Gauss equation we have
\[ \mathrm{Ric}_B(\nabla_B f) = \langle A_{\mathrm{H}_B}(\nabla_B f) ,
\nabla_B f \rangle - \sum_{i=1}^{\mathrm{dim}(B)} \| \alpha(\nabla_B f,
E_i)\|^2 \, .\]
Then from Theorem \ref{formulae} we get \[ 0 = \langle
\frac{\alpha_B(\nabla_B f, \nabla_B f)}{1 + \|\nabla_B f\|^2}  ,
\alpha(\nabla_B f,\nabla_B f) \rangle  - \sum_{i=1}^{\mathrm{dim}(B)} \|
\alpha(\nabla_B f, E_i)\|^2 \, .\]
Setting $E_1 := \frac{\nabla_B f}{\|\nabla_B f\|} $ we get
\[ 0 = \frac{\|\alpha(\nabla_B f, \nabla_B f) \|^2}{1 + \|\nabla_B f \|^2} -
\frac{\|\alpha(\nabla_B f, \nabla_B f) \|^2}{\| \nabla_B f \|^2} -
\sum_{i=2}^{\mathrm{dim}(B)} \| \alpha(\nabla_B f, E_i)\|^2 \]
and so
\[ 0 = \frac{-\|\alpha(\nabla_B f, \nabla_B f) \|^2}{(1 + \|\nabla_B
f \|^2)\|\nabla_B f \|^2} - \sum_{i=2}^{\mathrm{dim}(B)} \| \alpha(\nabla_B
f, E_i)\|^2 \, .\]

Thus, $\alpha_B(\nabla_B f, \nabla_B f) = \alpha_B(\nabla_B f, E_i) = 0$
for $i=2,\cdots,\mathrm{dim}(B)$. Then
$\nabla_B f$ is in the nullity of the second fundamental form.
By Theorem \ref{formulae}, $(B,g)$ is minimal.
Then $B$ is a
minimal submanifold with non-negative Ricci tensor. It follows that $B$ is a
totally geodesic submanifold. Since, $f$ is eikonal and harmonic in $(B,g)$
with $B$ an Euclidean space we have that $f$ is a linear function and
so its graph over $B$ is other Euclidean space, i.e.
$M=\phi(B)$ is a totally geodesic submanifold. $\Box$

%
%
%
%

\section{Helix hypersurfaces with constant mean curvature}

In this section we give a proof of the following theorem which generalize Corollary 4.2 in \cite{GPR}.
For the proof we need the following corollary of the maximum principle for harmonic maps in \cite[Theorem 2]{Sam}.
\begin{lema} Let $f:M \rightarrow N$ be a harmonic map between the Riemannian manifolds $M,N$.
Assume that $f(M)$ is contained in the hypersurface $H \subset N$. If the shape operator of $H$ is definite
then $f$ is a constant map.
\end{lema}

{\bf Proof of Theorem \ref{meancurvature} }. \rm If the helix angle is zero then it is clear that the hypersurface is a cylinder.
So assume that the constant angle is different from zero. So a normal vector is not perpendicular the constant direction $\d$.
Observe that the subset $H$ of the sphere consisting of vectors whose angle with a fix vector $\d$ is constant different from $\frac{\pi}{2}$
is a totally umbilical non-totally geodesic submanifold. Hence the shape operator of $H$ is definite.
Now by Ruh-Vilms' theorem \cite{RuVi} the Gauss map of our helix surface is harmonic.
By the previous observation the image of such Gauss map is contained in the hypersurface $H$.
Then by the above lemma the Gauss map is constant. Hence the helix hypersurface is an open subset of some hyperplane. $\Box$\\

Unfortunately the above idea does not work for higher codimensional helix submanifolds. Let us explain where is the problem.
Let $\mathrm{G}(n,r)$ be the Grassmanian of $r$-planes in $\mathbb{R}^n$. For $\d \in \mathbb{R}^n $
define $H(\d, \theta) \subset \mathrm{G}(n,r)$ as the subset of $r$-planes whose angle with $\d$ is $\theta$.
Notice that for $\theta \neq 0$ the subset $H(\d, \theta) \subset \mathrm{G}(n,r)$ is a smooth hypersurface.
It is not difficult to see that $H(\d, \theta) \subset \mathrm{G}(n,r)$ is an orbit of the natural action of the
subgroup $SO(n)_{\d}$ of $SO(n)$ which leaves $\d$ fixed, i.e. the isotropy subgroup of $\d$.
The subgroup $SO(n)_{\d}$ is symmetric in $SO(n)$. Indeed, it is the fixed subgroup associated to the
involution $\sigma$ of $SO(n)$ induced by the symmetry with respect to the hyperplane $\langle \d, \cdot \rangle = 0$ in $\R^n$.
The principal curvatures of the orbits  $H(\d, \theta) \subset \mathrm{G}(n,r)$ were computed in \cite[p.65, Proposition 6]{Ver}.
So we see that unless the Grassmanian $\mathrm{G}(n,r)$ is a projective space the shape operator of the hypersurfaces
$H(\d, \theta) \subset \mathrm{G}(n,r)$ is never definite. Notice that the dimension of $H(\d, 0)$ is $(r-1)(n-r)$ so
if the codimension $n-r$ is greater than one $H(\d, 0)$ is not a hypersurface of $\mathrm{G}(n,r)$. Finally, in
codimension one the hypersurface $H(\d, 0)$ is totally geodesic hence its shape operator is non-definite.
So this explains the existence of non-totally geodesic cylinders over hypersurfaces with constant mean curvature.

\section*{Acknowledgements}
We would like to thank Francisco Vittone for several useful comments.
The second author thanks the hospitality of DISMA at Politecnico di Torino where this work was started.

\noindent {\bf Authors' Addresses:}

\vspace{.5cm}

\noindent A. J. Di Scala, \\ Dipartimento di Scienze Matematiche, Politecnico di Torino, \\
Corso Duca degli Abruzzi 24, 10129 Torino, Italy \\
\href{mailto:antonio.discala@polito.it}{antonio.discala@polito.it}\\
\url{http://calvino.polito.it/~adiscala/}

\vspace{.5cm}

\noindent G. Ruiz-Hern\'andez, \\ Instituto de Matem\'aticas, Universidad Nacional Aut\'onoma de M\'exico, \\
Ciudad Universitaria, C.P. 04510 D.F. M\'exico \\
\href{mailto:gruiz@matem.unam.mx}{gruiz@matem.unam.mx}\\

\end{document}